\documentclass[12pt, reqno]{amsart}
\usepackage[margin=1.0in]{geometry}
\usepackage{amsfonts}
\usepackage{amsmath}
\usepackage{amssymb}
\usepackage{setspace}
\usepackage{graphicx}
\usepackage{amscd,amssymb,amsthm}

\usepackage{hyperref}
\usepackage{graphicx}
\usepackage{amsmath, amsthm, amscd, amsfonts, amssymb, graphicx, color}
 \usepackage{url}
\usepackage{enumerate}
 \makeatletter
\let\reftagform@=\tagform@
\def\tagform@#1{\maketag@@@{(\ignorespaces\textcolor{blue}{#1}\unskip\@@italiccorr)}}
\renewcommand{\eqref}[1]{\textup{\reftagform@{\ref{#1}}}}
\makeatother
\usepackage{hyperref}
\hypersetup{colorlinks=true, linkcolor=red, anchorcolor=green,
citecolor=cyan, urlcolor=red, filecolor=magenta, pdftoolbar=true}

\usepackage{epsfig}        % For postscript
\usepackage{epic,eepic}       % For epic and eepic output from xfig

\newtheorem{theorem}{Theorem}
\theoremstyle{plain}

\newtheorem{corollary}{Corollary}

\newtheorem{lemma}{Lemma}

\newtheorem{remark}{Remark}

\numberwithin{equation}{section}

 \DeclareMathOperator{\osc}{Osc}
  
\begin{document}

\title[Two-Point Quadrature Rules for the Riemann--Stieltjes Integral]{ Two-Point Quadrature Rules for  Riemann--Stieltjes Integrals with  $L^p$--error estimates }

\author[M.W. Alomari]{M.W. Alomari}

\address{Department of Mathematics, Faculty of Science and
Information Technology, Irbid National University, 2600 Irbid
21110, Jordan.} \email{mwomath@gmail.com}
\date{\today}
\subjclass[2010]{41A55, 65D30, 65D32.}

\keywords{Quadrature  formula, Riemann-Stieltjes integral,
Ostrowski's inequality}

\begin{abstract}In this work, we construct a new general two-point quadrature rules for the Riemann--Stieltjes integral $\int_a^b {f\left( t \right)du\left( t
\right)}$, where the integrand $f$ is assumed to be satisfied with
the H\"{o}lder condition on $[a,b]$ and the integrator $u$ is of
bounded variation on $[a,b]$. The dual formulas under the same
assumption are proved. Some sharp error $L^p$--Error estimates for
the proposed quadrature rules are also obtained.
\end{abstract}

\maketitle

%=============================================================================
\section{Introduction}
%=============================================================================

The number of proposed quadrature rules that provides
approximation for the  Riemann--Stieltjes integral  ($
\mathcal{RS}$--integral) $\int_a^b {f\left( t
    \right)du\left( t \right)}$  using derivatives or without using
derivatives are very rare in comparison with the large number of
methods available to approximate the classical Riemann integral $
\int_{a}^{b}{f\left( t\right) dt}$.

The problem of introducing quadrature rules for
$\mathcal{RS}$-integral $\int_a^b {fdg}$ was studied via theory of
inequalities by many authors. Two famous real inequalities were
used in this approach, which are the well known Ostrowski  and
Hermite-Hadamard inequalities and their modifications. For this
purpose and in order to approximate the $\mathcal{RS}$-integral
$\int_a^b {f\left( t \right)du\left( t \right)}$,   a
generalization of closed Newton-Cotes quadrature rules of
$\mathcal{RS}$-integrals without using derivatives provides a
simple and robust solution to a significant problem in the
evaluation of certain applied probability models was presented by
Tortorella in \cite{Tortorella}.

In 2000, Dragomir \cite{Dragomir2} introduced the Ostrowski's
approximation formula (which is of One-point type formula) as
follows:
\begin{align*}
\int_a^b {f\left( t \right)du\left( t \right)}\cong f\left( x
\right) \left[ {u\left( {b} \right) - u\left( a \right)} \right]
\qquad \forall x\in [a,b].
\end{align*}
Several error estimations for this approximation had been done in
the works \cite{Dragomir1} and \cite{Dragomir2}.

From different point of view, the authors of \cite{Dragomir5} (see
also \cite{Barnett,Barnett1})  considered the problem of
approximating the Stieltjes integral $\int_a^b {f\left( t
    \right)du\left( t \right)}$ via the generalized trapezoid formula:
\begin{align*}
\int_a^b {f\left( t \right)du\left( t \right)}\cong \left[
{u\left( x \right) - u\left( a \right)} \right]f\left( a \right) +
\left[ {u\left( b \right) - u\left( x \right)} \right]f\left( b
\right).
\end{align*}
Many authors have studied this quadrature rule under various
assumptions of integrands and integrators. For full history of
these two quadratures see \cite{alomari3} and the references
therein.

Another trapezoid type formula was considered in \cite{Dragomir8},
which reads:
\begin{align*}
\int_a^b {f\left( t \right)du\left( t \right)}\cong \frac{f\left(
    a \right) + f\left( b \right)}{2}\left[ {u\left( {b} \right) -
    u\left( a \right)} \right] \qquad \forall x\in [a,b].
\end{align*}
Some related results had been presented by the same author in
\cite{Dragomir6} and \cite{Dragomir7}. For other connected results
see \cite{CeroneDragomir} and \cite{CeroneDragomir1}.

In 2008, Mercer \cite{Mercer}  introduced the following trapezoid
type formula for the $\mathcal{RS}$-integral
\begin{align}
\label{Mercer.Q}\int_a^b {fdg} \cong \left[ {G - g\left( a
    \right)} \right]f\left( a \right) + \left[ {g\left( b \right) - G}
\right]f\left( b \right),
\end{align}
where $G= \frac{1}{{b  -a}}\int_a^b  {g\left( t \right)dt}$.

Recently, Alomari and Dragomir \cite{alomari1}, proved several new
error bounds for the Mercer--Trapezoid quadrature rule
(\ref{Mercer.Q}) for the $\mathcal{RS}$-integral under various
assumptions involved the integrand $f$ and the integrator $g$.
\newline

Follows Mercer approach in \cite{Mercer}, Alomari and Dragomir
\cite{alomari6} introduced the following three-point quadrature
formula:
\begin{align}
\label{error.term} \int_a^b {f\left( t \right)dg \left( t \right)}
&\cong \left[ {G\left( {a,x} \right) - g\left( a \right)}
\right]f\left( a \right) + \left[ {G\left( {x,b} \right) - G\left(
    {a,x} \right)} \right]f\left( x \right)
\nonumber\\
&\qquad+ \left[ {g\left( b \right) - G\left( {x,b} \right)}
\right]f\left( b \right)
\end{align}
for all $a<x<b$, where $G\left( {\alpha ,\beta } \right): =
\frac{1}{{\beta  - \alpha }}\int_\alpha ^\beta  {g\left( t
    \right)dt}$.

\noindent Several error estimations of Mercer's type quadrature
rules for $\mathcal{RS}$-integral under various assumptions about
the function involved have been considered in \cite{alomari1} and
\cite{alomari4}.

Motivated by Guessab-Schmeisser inequality (see \cite{Guessab})
which is of Ostrowski's type, Alomari in \cite{alomari2} and
\cite{alomari5} presented the following approximation formula for
$\mathcal{RS}$-integrals:
\begin{align}
\int_a^b {f\left( t \right)du\left( t \right)} \cong\left[
{u\left( {\frac{{a + b}}{2}} \right) - u\left( a
    \right)} \right]f\left( x \right) + \left[ {u\left( b \right) -
    u\left( {\frac{{a + b}}{2}} \right)} \right]f\left( {a + b - x}
\right),
\end{align}
for all $x \in \left[ {a,\frac{{a + b}}{2}}\right]$. For other
related results see \cite{alomari3}. For different approaches
variant quadrature formulae the reader may refer to \cite{AG},
\cite{M}, \cite{Gautschi} and  \cite{Munteanu}.

Among others the $L^{\infty}$-norm gives the highest possible
degree of precision; so that it is recommended to be `almost' the
norm of choice.  However, in some cases we cannot access the
$L^{\infty}$-norm, so that  $L^p$-norm  ($1\le p < \infty$) is
considered to be a variant norm  in error estimations.

In this work, several $L^p$-error estimates ($1\le p < \infty$) of
general two and three points quadrature rules for
Riemann-Stieltjes integrals are presented. The presented proofs
depend on new triangle type inequalities for
$\mathcal{RS}$-integrals.

 Let $f$ be defined on $[a,b]$. If $P :=
\left\{ {x_0, x_1, \cdots, x_n} \right\}$ is a partition of
$[a,b]$, write
$$\Delta f_i  = f\left( {x_i } \right) - f\left( {x_{i - 1} }
\right),$$ for $i=1,2, \cdots, n$. A function $f$ is said to be of
bounded $p$-variation  if there exists a positive number $M$ such
that $\left( {\sum\limits_{i = 1}^n {\left| {\Delta f_i }
\right|^p} } \right)^{\frac{1}{p}}  \le M$, $(1 \le p < \infty)$
for all partition of $[a,b]$, (see \cite{Kolyada}).

Let $f$ be of bounded $p$-variation on $[a,b]$, and let $\sum (P)$
denote the sum $\left( {\sum\limits_{i = 1}^n {\left| {\Delta f_i
} \right|^p} } \right)^{\frac{1}{p}} $ corresponding to the
partition $P$ of $[a,b]$. The number
\begin{align*}
\bigvee_{a}^{b}(f;p) = \sup\left\{ {\sum (P): P \in\mathcal{P}
\left( \left[ a,b \right]  \right)} \right\},\,\,\,\,\,\,\,\,\,\,1
\le p < \infty
\end{align*}
is called the total $p$--variation of $f$ on the interval $[a,b]$,
where $\mathcal{P}{\left( \left[ a,b \right]  \right) }$ denotes
the set of all partitions of $[a,b]$. For $p = 1$ it is the usual
variation of $f(x)$ that was introduced by Jordan (see
\cite{Golubov1}, \cite{Golubov2}). For very constructive
systematic study of Jordan variation we recommend the interested
reader to refer to \cite{Natanson}.

In special case, we define the variation of order $\infty$ of $f$
  along $[a,b]$ in the classical sense,
i.e., if there exists a positive number $M$ such that
\begin{align*}
\sum\limits_{i = 1}^n {  \osc\left( {f;\left[ {x_{i-1}^{\left( n
\right)} ,x_{i}^{\left( n \right)} } \right]}
\right)}=\sum\limits_{i = 1}^n {\left({  \sup  -  \inf
  }\right)  f\left( t_i \right) } \le M, \qquad  t_i  \in \left[ {x_{i-1}^{\left( n \right)} ,x_{i}^{\left( n
\right)} } \right],
\end{align*}
for all partition of $[a,b]$, then $f$ is said to be of bounded
$\infty$--variation on $[a,b]$. The number
\begin{align*}
\bigvee_a^b(f;\infty) = \sup\left\{ {\sum (P): P \in
\mathcal{P}{[a,b]}} \right\}:=\osc\left( {f;\left[ {a,b} \right]}
\right),
\end{align*}
is called the oscillation of $f$ on $[a,b]$. Equivalently, we may
define the oscillation of $f$ as, (see \cite{Dudley}):
\begin{align*}
\bigvee_a^b(f;\infty)= \mathop {\lim }\limits_{p \to \infty
}\bigvee_a^b(f;p)&= \mathop {\sup }\limits_{x \in \left[ {a,b}
\right]} \left\{ {f\left( x \right)} \right\} - \mathop {\inf
}\limits_{x \in \left[ {a,b} \right]} \left\{ {f\left( x \right)}
\right\}
\\
&= \osc\left( {f;\left[ {a,b} \right]} \right).
\end{align*}

Let $\mathcal{W}_p$ denotes the class of all functions of bounded
$p$-variation $(1\le p \le \infty)$. For an arbitrary $p \ge 1$
the class $\mathcal{W}_p$ was firstly introduced by Wiener in
\cite{Wiener}, where he had shown that $\mathcal{W}_p$ can only
have discontinuities of the first kind. More generally, if $f$ is
a real function of bounded $p$-variation on an interval $[a,b]$,
then:
\begin{itemize}
\item $f$ is bounded, and
\begin{align*}
\osc\left( {f;\left[ {a,b} \right]} \right) \le
\bigvee_a^b(f;p)\le \bigvee_a^b(f;1).
\end{align*}
This fact follows by Jensen's inequality applied for
$h(p)=\bigvee_a^b(f;p)$ which is log-convex and decreasing for all
$p>1$. Moreover,   the
 inclusions
\begin{align*}
\mathcal{W}_{\infty}(f) \subset \mathcal{W}_q(f) \subset
\mathcal{W}_p(f)\subset \mathcal{W}_1(f)
\end{align*}
are valid for all $1 < p < q < \infty$, (see \cite{Young}).

\item $f$ is continuous except at most on a countable set.

\item $f$ has one-sided limits everywhere (limits from the left
everywhere in $(a,b]$, and from the right everywhere in $[a,b)$;

\item The derivative $f'(x)$ exists almost everywhere (i.e. except
for a set of measure zero).

\item If $f\left(x\right)$ is differentiable on $\left[ a,b
\right]$, then
\begin{align*}
\bigvee_a^b\left( {f;p} \right)  = \left( {\int_a^b {\left|
{f^{\prime}\left( t \right)} \right|^p dt} } \right)^{
\frac{1}{p}}=\left\| f^{\prime} \right\|_p, \qquad   1\le p <
\infty.
\end{align*}
\end{itemize}
 \begin{lemma}\cite{A}
 \label{lemma1}
\emph{Fix $1 \le p < \infty$. Let $f,g : [a,b] \to \mathbb{R}$ be
such that $f$ is continuous on $[a,b]$ and $g$ is of bounded
$p$--variation on $[a,b]$. Then the Riemann--Stieltjes integral
$\int_a^b {f\left( t \right)dg\left( t \right)}$ exists and the
inequality:
\begin{align}
\label{eq1.4}\left| {\int_a^b {w\left( t \right)d\nu\left( t
\right)} } \right| \le \left\|{w}\right\|_{\infty}\cdot \osc\left(
{\nu;\left[ {a,b} \right]} \right)\le
\left\|{w}\right\|_{\infty}\cdot\bigvee_a^b\left( {\nu;p} \right),
\end{align}
holds.  The constant $`1$' in the both inequalities is the best
possible. }
\end{lemma}

\begin{lemma}\cite{A}\label{lemma2}
Let $1 \le p < \infty$. Let $w,\nu : [a,b] \to \mathbb{R}$ be such
that is $w \in L^p[a,b]$ and $\nu$ has a Lipschitz property on
$[a,b]$. Then the inequality
\begin{align}
\label{eq1.5}\left| {\int_a^b {w\left( t \right)d\nu\left( t
\right)} } \right| \le L \left( {b-a} \right)^{1 - {\textstyle{1
\over p}}}\cdot  \left\| w \right\|_p,
\end{align}
holds and the constant $`1$' in the right hand side is the best
possible, where
\begin{align*}
\left\| w \right\|_p  = \left( {\int_a^b {\left| {w\left( t
\right)} \right|^p dt} } \right)^{1/p}, \qquad (1\le p < \infty).
\end{align*}
\end{lemma}

In this paper, we establish  two--point of Ostrowski's integral
inequality for the Riemann-Stieltjes integral $\int_a^b {f\left( t
\right)du\left( t \right)}$, where $f$ is assumed to be of
$r$-$H$-H\"{o}lder type on $[a,b]$ and $u$ is of bounded variation
on $[a,b]$, are given.  The dual formulas under the same
assumption are proved. Some sharp error $L^p$--Error estimates for
the proposed quadrature rules are also obtained.

%========================================================================================
\section{The Results}
%========================================================================================
Consider the quadrature rule
\begin{align}
\int_a^b{f\left( s \right)du\left( s \right)}=
\mathcal{Q}^{[a,b]}\left(f,u;t_0,x,t_1\right) +
\mathcal{R}^{[a,b]}\left(f,u;t_0,x,t_1\right)
\end{align}
where $\mathcal{Q}^{[a,b]}\left(f,u;t_0,x,t_1\right)$ is the
quadrature formula
\begin{align}
\mathcal{Q}^{[a,b]}\left(f,u;t_0,x,t_1\right)=\left[ {u\left( x
\right) - u\left( a \right)} \right]f\left( {t_0 } \right) +
\left[ {u\left( b \right) -
    u\left( x \right)} \right]f\left( {t_1 } \right)
\end{align}
for all $a\le t_0 \le x \le t_1  \le b$.

Hence, the remainder term
$\mathcal{R}^{[a,b]}\left(f,u;t_0,x,t_1\right)$ is given by
\begin{align}
\mathcal{R}^{[a,b]}\left(f,u;t_0,x,t_1\right):=  \int_a^b {f\left(
s \right)du\left( s \right)}  - \left[ {u\left( x \right) -
u\left( a \right)}
    \right]f\left( {t_0 } \right) - \left[ {u\left( b \right) -
        u\left( x \right)} \right]f\left( {t_1 } \right)
\end{align}
The following Two-point Ostrowski's inequality for
Riemann-Stieltjes integral holds.
\begin{theorem}\label{thm1}
Let $f:[a,b]\rightarrow \mathbb{R}$  be H\"{o}lder continuous of
order $r$, $(0 < r  \le 1)$, and $u:[a,b]\rightarrow \mathbb{R}$
is a mapping of bounded $p$-variation $(1\le p \le \infty)$ on
$[a, b]$. Then we have the inequality
\begin{align}
&\left| {\mathcal{R}^{[a,b]}\left(f,u;t_0,x,t_1\right)} \right|
\label{eq2.4}
\\
&\le H \max \left\{ {\left[ {\frac{{x - a}}{2} + \left| {t_0 -
\frac{{a + x}}{2}} \right|} \right],\left[ {\frac{{b - x}}{2} +
\left| {t_1 - \frac{{x + b}}{2}} \right|} \right]} \right\}^r
\cdot \bigvee_a^b\left( {u;p} \right)\nonumber
\end{align}
for all  $a\le t_0 \le x \le t_1 \le b$. Furthermore, the first
half of each max-term is the best possible in the sense that it
cannot be replaced by a smaller one, for all $r \in (0,1]$.
\end{theorem}

\begin{proof}
Using the integration by parts formula for Riemann--Stieltjes
integral, we have
\begin{align*}
&\int_a^x {\left[ {f\left( {t_0 } \right) - f\left( s \right)}
\right]du\left( s \right)}  + \int_x^b {\left[ {f\left( {t_1 }
\right) - f\left( s \right)} \right]du\left( s \right)}
\\
&= \int_a^x {f\left( {t_0 } \right)du\left( s \right)}  + \int_x^b
{f\left( {t_1 } \right)du\left( s \right)}  - \int_a^b {f\left( s
\right)du\left( s \right)}
\\
&=\left[ {u\left( x \right) - u\left( a \right)} \right]f\left(
{t_0 } \right) + \left[ {u\left( b \right) - u\left( x \right)}
\right]f\left( {t_1 } \right) - \int_a^b {f\left( s
\right)du\left( s \right)}
\\
&=-\mathcal{R}^{[a,b]}\left(f,u;t_0,x,t_1\right),
\end{align*}
It is well known that if $p : [c, d] \to \mathbb{R}$ is continuous
and $\nu : [c, d] \to \mathbb{R}$ is of $p$-bounded variation
($1\le p<\infty$), then the Riemann-Stieltjes integral
$\int_c^d{p(t)d\nu (t)}$ exists and the following inequality
holds:
\begin{align}
\label{eq2.5}\left| {\int_c^d {p\left( t \right)d\nu\left( t
\right)} } \right| \le \mathop {\sup }\limits_{t \in \left[ {c,d}
\right]} \left| {p\left( t \right)} \right| \bigvee_c^d\left( \nu
\right).
\end{align}
Applying the inequality \eqref{eq2.5} for $\nu(t) = u(t)$, $p(t)=
f\left( t_0 \right) - f\left( s \right)$, for all $s \in \left[
{a,x} \right]$; and then for $p(t)=f\left( {t_1} \right) - f\left(
s \right)$, $\nu(t) = u(t)$ for all $t \in \left( {x,b} \right]$,
we get
\begin{align}
&\left| {\left[ {u\left( x \right) - u\left( a \right)}
\right]f\left( {t_0 } \right) + \left[ {u\left( b \right) -
u\left( x \right)} \right]f\left( {t_1 } \right) - \int_a^b
{f\left( s \right)du\left( s \right)}} \right|
\nonumber\\
&=\left| {\int_a^x {\left[ {f\left( {t_0 } \right) - f\left( s
\right)} \right]du\left( s \right)}  + \int_x^b {\left[ {f\left(
{t_1 } \right) - f\left( s \right)} \right]du\left( s \right)}}
\right|
\nonumber\\
&\le \left| {\int_a^x {\left[ {f\left( {t_0 } \right) - f\left( s
\right)} \right]du\left( s \right)}  } \right| + \left| { \int_x^b
{\left[ {f\left( {t_1 } \right) - f\left( s \right)}
\right]du\left( s \right)} } \right|
\nonumber\\
&\le \mathop {\sup }\limits_{s \in \left[ {a,x} \right]} \left|
{f\left( {t_0} \right) - f\left( s \right)} \right| \cdot
\bigvee_a^{x}\left( {u;p} \right) + \mathop {\sup }\limits_{s \in
\left[ {x,b} \right]} \left| {f\left( {t_1} \right) - f\left( s
\right)} \right| \cdot \bigvee_{x}^b\left( {u;p}
\right).\label{eq2.6}
\end{align}
As $f$ is of $r$-$H$-H\"{o}lder type, we have
\begin{align*}
\mathop {\sup }\limits_{s \in \left[ {a,x} \right]} \left|
{f\left( {t_0} \right) - f\left( s \right)} \right| &\le \mathop
{\sup }\limits_{s \in \left[ {a,x} \right]} \left[ {H\left| {t_0 -
s} \right|^r } \right]
\\
&= H\max \left\{ {\left( {x- t_0} \right)^r ,\left( {t_0-a}
\right)^r } \right\}
\\
&= H\left[ {\max \left\{ {\left( {x - t_0} \right),\left( {t_0-a}
\right)} \right\}} \right]^r
\\
&= H\left[ {\frac{{x - a}}{2} + \left| {t_0 - \frac{{a + x}}{2}}
\right|} \right]^r,
\end{align*}
and
\begin{align*}
\mathop {\sup }\limits_{s \in \left[ {x,b} \right]} \left|
{f\left( {t_1} \right) - f\left( s \right)} \right| &\le \mathop
{\sup }\limits_{s \in \left[ {x,b} \right]} \left[ {H\left| {t_1 -
s} \right|^r } \right]
\\
&= H\max \left\{ {\left( {t_1-x} \right)^r ,\left( {b-t_1}
\right)^r } \right\}
\\
&= H\left[ {\max \left\{ {\left( {t_1 - x} \right),\left( {b -
t_1} \right)} \right\}} \right]^r
\\
&= H\left[ {\frac{{b - x}}{2} + \left| {t_1 - \frac{{x + b}}{2}}
\right|} \right]^r.
\end{align*}
Therefore, by \eqref{eq2.6}, we have
\begin{align*}
&\left| {\left[ {u\left( x \right) - u\left( a \right)}
\right]f\left( {t_0 } \right) + \left[ {u\left( b \right) -
u\left( x \right)} \right]f\left( {t_1 } \right) - \int_a^b
{f\left( s \right)du\left( s \right)}} \right|
\\
&\le H\left[ {\frac{{x - a}}{2} + \left| {t_0 - \frac{{a + x}}{2}}
\right|} \right]^r \cdot \bigvee_a^{x}\left( {u;p} \right) +
H\left[ {\frac{{b - x}}{2} + \left| {t_1 - \frac{{x + b}}{2}}
\right|} \right]^r\cdot \bigvee_{x}^b\left( {u;p} \right)
\\
&\le H \max \left\{ {\left[ {\frac{{x - a}}{2} + \left| {t_0 -
\frac{{a + x}}{2}} \right|} \right]^r,\left[ {\frac{{b - x}}{2} +
\left| {t_1 - \frac{{x + b}}{2}} \right|} \right]^r} \right\}
\cdot \bigvee_a^b\left( {u;p} \right)
\\
&= H \max \left\{ {\left[ {\frac{{x - a}}{2} + \left| {t_0 -
\frac{{a + x}}{2}} \right|} \right],\left[ {\frac{{b - x}}{2} +
\left| {t_1 - \frac{{x + b}}{2}} \right|} \right]} \right\}^r
\cdot \bigvee_a^b\left( {u;p} \right)
\end{align*}
To prove the sharpness of the constant $\frac{1}{2^r}$ for any $r
\in (0,1]$, assume that \eqref{eq2.4} holds with a constant $C >
0$, that is,
\begin{multline}
\label{eq2.7}\left| {\left[ {u\left( x \right) - u\left( a
\right)} \right]f\left( {t_0 } \right) + \left[ {u\left( b \right)
- u\left( x \right)} \right]f\left( {t_1 } \right) - \int_a^b
{f\left( s \right)du\left( s \right)}} \right|
\\
\le H \max \left\{ {\left[ {C\left( {x-a} \right) + \left| {t_0 -
\frac{{a + x}}{2}} \right|} \right],\left[ {C\left( {b-x} \right)
+ \left| {t_1 - \frac{{x + b}}{2}} \right|} \right]} \right\}^r
\cdot \bigvee_a^b\left( {u;p} \right).
\end{multline}
Choose $f(t)= t^r$, $r \in (0,1]$, $t \in [0,1]$ and $u : [0,1]
\to [0,\infty)$ given by
\begin{align*}
u\left( t \right) = \left\{ \begin{array}{l}
 0,\,\,\,\,\,\,\,\,\,\,\,t \in \left( {0,1} \right] \\
 \\
 -1,\,\,\,\,\,\,\,t = 0 \\
 \end{array} \right.
\end{align*}
As
\begin{align*}
\left| {f\left( x \right) - f\left( y \right)} \right| = \left|
{x^r  - y^r } \right| \le \left| {x - y} \right|^r, \,\, \forall x
\in [0,1], \,\,r \in (0,1],
\end{align*}
it follows that $f$ is $r$-$H$-H\"{o}lder type with the constant
$H = 1$.

By using the integration by parts formula for Riemann-Stieltjes
integrals, we have:
\begin{align*}
\int_0^1 {f\left( t \right)du\left( t \right)}  = f\left( 1
\right)u\left( 1 \right) - f\left( 0 \right)u\left( 0 \right) -
\int_0^1 {u\left( t \right)df\left( t \right)}=   0,
\end{align*}
and $\bigvee_0^1(u;p) = 1$. Consequently, by \eqref{eq2.7}, we get
\begin{align*}
\left| {t_0^r} \right| \le \max \left\{ {\left[ {Cx + \left| {t_0
- \frac{{x}}{2}} \right|} \right],\left[ {C\left( {1-x} \right) +
\left| {t_1 - \frac{{x + 1}}{2}} \right|} \right]} \right\}^r,
\,\,\, \forall t_0 \in \left[ {0,1} \right].
\end{align*}
For $t_0 = \frac{x}{2}$ and $t_1=x=1$ we get $\frac{1}{{2^r }} \le
C^r $, which implies that $C \ge \frac{1}{2}$.

It remains to prove the second part, so we consider
\begin{align*}
u\left( t \right) = \left\{ \begin{array}{l}
 0,\,\,\,\,\,\,\,\,\,\,\,t \in \left[ {0,1} \right) \\
 \\
 1,\,\,\,\,\,\,\,t = 1 \\
 \end{array} \right.
\end{align*}
therefore as we have obtained previously
\begin{align*}
\int_0^1 {f\left( t \right)du\left( t \right)} =
0\,\,\,\,\text{and}\,\,\,\,\bigvee_0^1(u;p) = 1.
\end{align*}
Consequently, by \eqref{eq2.4}, we get
\begin{align*}
\left| {t_1^r} \right| \le \max \left\{ {\left[ {Cx + \left| {t_0
- \frac{{x}}{2}} \right|} \right],\left[ {C\left( {1-x} \right) +
\left| {t_1 - \frac{{x + 1}}{2}} \right|} \right]} \right\}^r,
\,\,\, \forall t_0 \in \left[ {0,1} \right].
\end{align*}
For $t_0 = x=0$ and $t_1=\frac{1}{2}$ we get $\frac{1}{{2^r }} \le
C^r $, which implies that $C \ge \frac{1}{2}$, and the theorem is
completely proved.

\end{proof}

The following inequalities are hold:
\begin{corollary}
Let $f$ and $u$ as in Theorem \ref{thm1}. In \ref{eq2.4} choose
\begin{enumerate}
\item $t_0 = a$ and $t_1=b$, then we get the following trapezoid
type inequality
\begin{align*}
\left| {\mathcal{R}^{[a,b]}\left(f,u;a,x,b\right)} \right| \le H
\left[ {\frac{{b - a}}{2} + \left| {x - \frac{{a + b}}{2}}
\right|} \right]^r  \cdot \bigvee_a^b\left( {u;p} \right).
\end{align*}
or equivalently, we may write using parts formula for
Riemann-Stieltjes integral
\begin{align*}
\left| {\left[ {f\left( b \right) - f\left( a \right)}
\right]u\left( x \right) - \int_a^b {u\left( s \right)df\left( s
\right)}} \right| \le H \left[ {\frac{{b - a}}{2} + \left| {x -
\frac{{a + b}}{2}} \right|} \right]^r  \cdot \bigvee_a^b\left(
{u;p} \right).
\end{align*}
The constant $\frac{1}{2}$ is the best possible for all $r \in
(0,1]$.\\

\item $x = \frac{a+b}{2}$, then we get the following mid-point
type inequality
\begin{multline*}
\left| {\mathcal{R}^{[a,b]}\left(f,u;t_0,\frac{a+b}{2},t_1\right)
} \right|
\\
\le H \max \left\{ {\left[ {\frac{{b - a}}{4} + \left| {t_0 -
\frac{{3a + b}}{4}} \right|} \right],\left[ {\frac{{b - a}}{4} +
\left| {t_1 - \frac{{a + 3b}}{4}} \right|} \right]} \right\}^r
\cdot \bigvee_a^b\left( {u;p} \right).
\end{multline*}
The constant $\frac{1}{4}$ is the best possible for all $r \in
(0,1]$. For instance, setting $t_0=y$ and $t_1=a+b-y$, we get
\begin{align*}
\left| {\mathcal{R}^{[a,b]}\left(f,u;y,\frac{a+b}{2},a+b-y\right)}
\right| \le H \left[ {\frac{{b - a}}{4} + \left| {y - \frac{{3a +
b}}{4}} \right|} \right]^r \cdot \bigvee_a^b\left( {u;p} \right).
\end{align*}
for all $y\in \left[{a,\frac{a+b}{2}}\right]$.\\

\item $t_0=\frac{a+x}{2}$ and $t_1=\frac{x+b}{2}$, then
\begin{align*}
\left|
{\mathcal{R}^{[a,b]}\left(f,u;\frac{a+x}{2},x,\frac{x+b}{2}\right)}
\right| \le \frac{H}{2^r}\left[ {\frac{{b - a}}{2} + \left| {x -
\frac{{a + b}}{2}} \right|} \right]^r \cdot \bigvee_a^b\left(
{u;p} \right)
\end{align*}
Both constants $\frac{1}{2^r}$ and $\frac{1}{2}$ are the best
possible for all $r \in (0,1]$.
\end{enumerate}
\end{corollary}

\begin{corollary}
\label{cor4}Let $f$ be a   H\"{o}lder continuous function of order
$r$ $(0<r\le 1)$, on $[a,b]$, and $g : [a,b] \to \mathbb{R}$ is
continuous on $[a,b]$. Then we have the inequality
\begin{multline*}
 \left| { f\left( {t_0 } \right)\int_a^x
{g\left( s \right)ds}  + f\left( {t_1 } \right)\int_x^b {g\left( s
\right)ds}  - \int_a^b {f\left( s \right)g\left( s \right)ds} }
\right|
\\
\le H \max \left\{ {\left[ {\frac{{x - a}}{2} + \left| {t_0 -
\frac{{a + x}}{2}} \right|} \right],\left[ {\frac{{b - x}}{2} +
\left| {t_1 - \frac{{x + b}}{2}} \right|} \right]} \right\}^r
\cdot \left\| g \right\|_p,
\end{multline*}
for all $a \le t_0 \le x \le t_1 \le b$, where $\left\| g
\right\|_p = \left({\int_a^b {\left| {g\left( t \right)}
\right|^pdt} }\right)^{1/p}$.
\end{corollary}

\begin{proof}
Define the mapping $u:[a,b] \to \mathbb{R}$, $u(t) =
\int_a^t{g(s)ds}$. Then $u$ is differentiable on $(a,b)$ and
$u'(t) = g(t)$. Using the properties of the Riemann-Stieltjes
integral, we have
\begin{equation*}
\int_a^b {f\left( t \right)du\left( t \right)}  = \int_a^b
{f\left( t \right)g\left( t \right)dt},
\end{equation*}
and
\begin{equation*}
\bigvee_a^b\left( u;p \right) = \left({\int_a^b {\left| {u'\left(
t \right)} \right|^pdt} }\right)^{1/p}  = \left({\int_a^b {\left|
{g\left( t \right)} \right|^pdt} }\right)^{1/p},
\end{equation*}
which gives the required result.
\end{proof}

\begin{theorem}
\label{thm2}Let $1 \le p < \infty$. Let $f,u : [a,b] \to
\mathbb{R}$ be such that is $f \in L^p[a,b]$ and $u$ has a
Lipschitz property on $[a,b]$. If $f$ is $r$--$H$--H\"{o}lder
continuous, then the inequality
\begin{multline}
\left| {\mathcal{R}^{[a,b]}\left(f,u;t_0,x,t_1\right) } \right|
\le H L\left[{\left(x-a\right)^{1-\frac{1}{p}} \left(
{\frac{{\left( {t_0 - a} \right)^{rp + 1}+ \left( {x - t_0 }
\right)^{rp + 1} }}{{rp + 1}}} \right)^{\frac{1}{p}} }\right.
\\
\left.{+ \left(b-x\right)^{1-\frac{1}{p}} \left( {\frac{{\left(
{t_1 - x} \right)^{rp + 1}  + \left( {b - t_1 } \right)^{rp + 1}
}}{{rp + 1}}} \right)^{\frac{1}{p}} }\right]\label{eq2.8}
\end{multline}
holds for all $p>1$ and $r\in \left( 0,1\right]$.
\end{theorem}

\begin{proof}
From Lemma \ref{lemma2} we have
\begin{align*}
&\left| {\left[ {u\left( x \right) - u\left( a \right)}
\right]f\left( {t_0 } \right) + \left[ {u\left( b \right) -
u\left( x \right)} \right]f\left( {t_1 } \right) - \int_a^b
{f\left( s \right)du\left( s \right)}} \right|
\nonumber\\
&=\left| {\int_a^x {\left[ {f\left( {t_0 } \right) - f\left( s
\right)} \right]du\left( s \right)}  + \int_x^b {\left[ {f\left(
{t_1 } \right) - f\left( s \right)} \right]du\left( s \right)}}
\right|
\nonumber\\
&\le \left| {\int_a^x {\left[ {f\left( {t_0 } \right) - f\left( s
\right)} \right]du\left( s \right)}  } \right| + \left| { \int_x^b
{\left[ {f\left( {t_1 } \right) - f\left( s \right)}
\right]du\left( s \right)} } \right|
\nonumber\\
&\le L\left[{\left(x-a\right)^{1-\frac{1}{p}}\left({\int_a^x
{\left| {f\left( {t_0 } \right) - f\left( s \right)}
\right|^pds}}\right)^{\frac{1}{p}}  } \right.
\nonumber\\
&\qquad\qquad \left.{+
\left(b-x\right)^{1-\frac{1}{p}}\left({\int_x^b {\left| {f\left(
{t_1 } \right) - f\left( s \right)} \right|^pds}
}\right)^{\frac{1}{p}}} \right]
\nonumber\\
&\le H L\left[{\left(x-a\right)^{1-\frac{1}{p}} \left({\int_a^x
{\left| {t_0-s} \right|^{rp}ds}}\right)^{\frac{1}{p}}  +
\left(b-x\right)^{1-\frac{1}{p}} \left({\int_x^b {\left| {t_1 - s}
\right|^{rp}ds} }\right)^{\frac{1}{p}}} \right]
\\
&=H L\left[{\left(x-a\right)^{1-\frac{1}{p}} \left( {\frac{{\left(
{t_0 - a} \right)^{rp + 1}+ \left( {x - t_0 } \right)^{rp + 1}
}}{{rp + 1}}} \right)^{\frac{1}{p}} }\right.
\\
&\qquad\qquad\left.{+ \left(b-x\right)^{1-\frac{1}{p}} \left(
{\frac{{\left( {t_1 - x} \right)^{rp + 1}  + \left( {b - t_1 }
\right)^{rp + 1} }}{{rp + 1}}} \right)^{\frac{1}{p}} }\right].
\end{align*}
which proves the required result.
\end{proof}

\begin{corollary}
    Let $f$ and $u$ as in Theorem \ref{thm2}. In
    \eqref{eq2.8} choose
    \begin{enumerate}
        \item $t_0 = a$ and $t_1=b$, then we get the following trapezoid
        type inequality
        \begin{align*}
        \left| {\mathcal{R}^{[a,b]}\left(f,u;a,x,b\right) } \right|
        \le H L\left[{\left(x-a\right)^{1-\frac{1}{p}} \left( {\frac{{   \left( {x - a } \right)^{rp + 1}
                }}{{rp + 1}}} \right)^{\frac{1}{p}}  + \left(b-x\right)^{1-\frac{1}{p}} \left(
            {\frac{{\left( {b - x} \right)^{rp + 1}    }}{{rp + 1}} } \right)^{\frac{1}{p}} }\right].
        \end{align*}
        or equivalently, we may write using parts formula for
        Riemann-Stieltjes integral
        \begin{multline*}
        \left| {\left[ {f\left( b \right) - f\left( a \right)}
            \right]u\left( x \right) - \int_a^b {u\left( s \right)df\left( s
                \right)}} \right|
        \\
        \le H L\left[{\left(x-a\right)^{1-\frac{1}{p}} \left( {\frac{{   \left( {x - a } \right)^{rp + 1}
                }}{{rp + 1}}} \right)^{\frac{1}{p}}  + \left(b-x\right)^{1-\frac{1}{p}} \left(
            {\frac{{\left( {b - x} \right)^{rp + 1}    }}{{rp + 1}} } \right)^{\frac{1}{p}} }\right].
        \end{multline*}

        \item $x = \frac{a+b}{2}$, then we get the following mid-point
        type inequality
        \begin{multline*}
        \left| { \mathcal{R}^{[a,b]}\left(f,u;t_0,\frac{a+b}{2},t_1\right)} \right|
        \\
        \le H L\left[{\left(\frac{b-a}{2}\right)^{1-\frac{1}{p}} \left( {\frac{{\left(
                        {t_0 - a} \right)^{rp + 1}+ \left( {\frac{a+b}{2} - t_0 } \right)^{rp + 1}
                }}{{rp + 1}}} \right)^{\frac{1}{p}} }\right.
        \\
     \left.{+ \left(\frac{b-a}{2}\right)^{1-\frac{1}{p}} \left(
            {\frac{{\left( {t_1 - \frac{a+b}{2}} \right)^{rp + 1}  + \left( {b - t_1 }
                        \right)^{rp + 1} }}{{rp + 1}}} \right)^{\frac{1}{p}} }\right].
        \end{multline*}
      For instance, setting $t_0=y$ and $t_1=a+b-y$, we get
        \begin{multline*}
        \left| { \mathcal{R}^{[a,b]}\left(f,u;y,\frac{a+b}{2},a+b-y\right) } \right|
        \\
        \le     2H L\left[{\left(\frac{b-a}{2}\right)^{1-\frac{1}{p}} \left( {\frac{{\left(
                        {t_0 - a} \right)^{rp + 1}+ \left( {\frac{a+b}{2} - t_0 } \right)^{rp + 1}
                }}{{rp + 1}}} \right)^{\frac{1}{p}}   }\right].
        \end{multline*}
        for all $y\in \left[{a,\frac{a+b}{2}}\right]$.\\

    \item $t_0=\frac{3a+b}{4}$, $x=\frac{a+b}{2}$ and $t_1=\frac{a+3b}{4}$, then
    \begin{align*}
 \left| { \mathcal{R}^{[a,b]}\left(f,u;\frac{3a+b}{4},\frac{a+b}{2},\frac{a +3b}{4}\right) } \right|
    \le     HL\frac{{\left( {b - a} \right)^{1 + r} }}{{2^{2r + \frac{1}{p}} \left( {rp + 1} \right)^{\frac{1}{p}} }}.
    \end{align*}

    \end{enumerate}
\end{corollary}
 Now, let $I$ be a real interval such that $[a,b] \subseteq I^{\circ}$
 the interior of $I$, $a,b\in \mathbb{R}$ with $a<b$. Consider
 $\mathfrak{U}^p(I)$ ($p>1$) be the space of all positive $n$-th
 differentiable functions $f$ whose $n$-th derivatives $f^{(n)}$ is
 positive locally absolutely continuous on $I^{\circ}$ with $\int_a^b {\left( {f^{(n)}\left( t \right)} \right)^p dt} < \infty,
 $ and $f^{(n)}(a)=f^{(n)}(b)=0$.

 $L^p$-error estimates for Riemann--Stieltjes  $\int_a^b {f\left( t
    \right)du\left( t \right)}$ where $f$ belongs  to  $\mathfrak{U}^p(I)$ is considered in the following result.
\begin{theorem}
\label{thm3}    Let $1 \le p < \infty$. Let $f,u : [a,b] \to
\mathbb{R}$ be such
    that is $f \in \mathfrak{U}^p(I)$ and $u$ has a Lipschitz property on
    $[a,b]$. If $f$ is $r$--$H$--H\"{o}lder continuous, then the inequality
        holds for all $p>1$ and $r\in \left( 0,1\right]$.
    \begin{multline}
\label{eq2.9}   \left|
{\mathcal{R}^{[a,b]}\left(f,u;t_0,x,t_1\right)} \right|
      \le  L        \left( {\frac{{p\sin \left( {\frac{\pi }{p}} \right)}}{{\pi \sqrt[p]{{p - 1}}}}} \right)^n \left\{ {\left(x-a\right)^{1-\frac{1}{p}}
  \left[ {\frac{{x - a}}{2} + \left| {t_0  - \frac{{x + a}}{2}} \right| } \right]^n   }\right.
\\
  \left.{+\left(b-x\right)^{1-\frac{1}{p}}
 \left[ {\frac{{b - x}}{2} + \left| {t_1  - \frac{{x + b}}{2}} \right| } \right]^n  }\right\}\left\| {f^{\left( n \right)} } \right\|_{p,\left[ {a,b} \right]}
\end{multline}
\end{theorem}

\begin{proof}
As in the proof of Theorem \ref{thm2}, we have by Lemma
\ref{lemma2}
    \begin{align*}
    &\left| {\left[ {u\left( x \right) - u\left( a \right)}
        \right]f\left( {t_0 } \right) + \left[ {u\left( b \right) -
            u\left( x \right)} \right]f\left( {t_1 } \right) - \int_a^b
        {f\left( s \right)du\left( s \right)}} \right|
    \nonumber\\
        &\le L\left[{\left(x-a\right)^{1-\frac{1}{p}}\left({\int_a^x
            {\left| {f\left( {t_0 } \right) - f\left( s \right)}
                \right|^pds}}\right)^{\frac{1}{p}}  } \right.
    \nonumber\\
    &\qquad\qquad \left.{+
        \left(b-x\right)^{1-\frac{1}{p}}\left({\int_x^b {\left| {f\left(
                    {t_1 } \right) - f\left( s \right)} \right|^pds}
        }\right)^{\frac{1}{p}}} \right]
    \nonumber\\
    &\le  L\left[{\left(x-a\right)^{1-\frac{1}{p}}
        \left( {\frac{{p\sin \left( {\frac{\pi }{p}} \right)}}{{\pi \sqrt[p]{{p - 1}}}}} \right)^n \left[ {\frac{{x - a}}{2} + \left| {t_0  - \frac{{x + a}}{2}} \right|^n } \right]\left\| {f^{\left( n \right)} } \right\|_{p,\left[ {a,x} \right]}  }\right.
        \\
        &\qquad\qquad \left.{ +\left(b-x\right)^{1-\frac{1}{p}}
        \left( {\frac{{p\sin \left( {\frac{\pi }{p}} \right)}}{{\pi \sqrt[p]{{p - 1}}}}} \right)^n \left[ {\frac{{b - x}}{2} + \left| {t_1  - \frac{{x + b}}{2}} \right|^n } \right]\left\| {f^{\left( n \right)} } \right\|_{p,\left[ {x,b} \right]}
         } \right]
     \\
      &\le  L       \left( {\frac{{p\sin \left( {\frac{\pi }{p}} \right)}}{{\pi \sqrt[p]{{p - 1}}}}} \right)^n \left\{{\left(x-a\right)^{1-\frac{1}{p}}
 \left[ {\frac{{x - a}}{2} + \left| {t_0  - \frac{{x + a}}{2}} \right|  } \right]^n   }\right.
    \\
    &\qquad\qquad  \left.{+\left(b-x\right)^{1-\frac{1}{p}}
\left[ {\frac{{b - x}}{2} + \left| {t_1  - \frac{{x + b}}{2}}
\right| } \right]^n  }\right\}\left\| {f^{\left( n \right)} }
\right\|_{p,\left[ {a,b} \right]}
    \end{align*}
    which proves the required result, where we have used that fact that
    if
    $h\in \mathfrak{U}^p(I)$   then for all $\xi \in (a,b)$ we have
    \begin{align}
    \label{moh}\int_a^b {\left| {h\left( t \right) - h\left( \xi
            \right)} \right|^p dt}
    \le\left( {\frac{{p^p \sin ^p \left(
                {{\textstyle{\pi  \over p}}} \right)}}{{\pi ^p \left( {p - 1}
                \right)}}} \right)^{n}\left[ {\frac{{b - a}}{2} + \left| {\xi  -
            \frac{{a + b}}{2}} \right|} \right]^{np} \cdot\int_a^b {\left(
        {h^{\left( n \right)} \left( x \right)} \right)^p dx}.
    \end{align}
    In case $n=1$, the inequality \eqref{moh} is sharp, see
    \cite{alomari}.
\end{proof}

\begin{remark}
If  $f \in \mathfrak{U}^p(I)$  and $f^{(n)}$ is bounded on  $I$,
so that as $p\to \infty $ in \eqref{eq2.9}, then since $\mathop
{\lim }\limits_{p \to \infty } \frac{{p\sin \left( {\frac{\pi
}{p}} \right)}}{{\sqrt[p]{{p - 1}}}} = \pi $, therefore we have
\begin{multline}
\left| {\mathcal{R}^{[a,b]}\left(f,u;t_0,x,t_1\right)} \right| \le
L   \left\{ {\left(x-a\right)  \left[ {\frac{{x - a}}{2} + \left|
{t_0  - \frac{{x + a}}{2}} \right| } \right]^n   }\right.
\\
\left.{+\left(b-x\right)    \left[ {\frac{{b - x}}{2} + \left|
{t_1  - \frac{{x + b}}{2}} \right| } \right]^n  }\right\}\left\|
{f^{\left( n \right)} } \right\|_{\infty,\left[ {a,b} \right]}
\end{multline}
\end{remark}

In what follows we observe several general quadrature rules for
the Riemann--Stieltjes integral  $\int_a^b {f\left( t
    \right)du\left( t \right)}$ where $f$ is $n$-times differentiable whose derivatives belongs ton $L^p([a,b])$. To the best of our knowledge, this is the first time  of such result concerning  Riemann--Stieltjes integral without using interpolation.
\begin{corollary}
Let $f$ and $u$ as in Theorem \ref{thm3}. In \eqref{eq2.9} choose
\begin{enumerate}
    \item $t_0 = a$ and $t_1=b$, then we get the following trapezoid
    type inequality
    \begin{align*}
    \left| {\mathcal{R}^{[a,b]}\left(f,u;a,x,b\right)} \right|
    \le L       \left( {\frac{{p\sin \left( {\frac{\pi }{p}} \right)}}{{\pi \sqrt[p]{{p - 1}}}}} \right)^n \left\{{\left(x-a\right)^{n+1-\frac{1}{p}} +\left(b-x\right)^{n+1-\frac{1}{p}}
  }\right\}\left\| {f^{\left( n \right)} } \right\|_{p,\left[ {a,b} \right]}.
    \end{align*}
    or equivalently, we may write using parts formula for
    Riemann-Stieltjes integral
    \begin{multline*}
    \left| {\left[ {f\left( b \right) - f\left( a \right)}
        \right]u\left( x \right) - \int_a^b {u\left( s \right)df\left( s
            \right)}} \right|
    \\
    \le L       \left( {\frac{{p\sin \left( {\frac{\pi }{p}} \right)}}{{\pi \sqrt[p]{{p - 1}}}}} \right)^n \left\{{\left(x-a\right)^{n+1-\frac{1}{p}} +\left(b-x\right)^{n+1-\frac{1}{p}}
    }\right\}\left\| {f^{\left( n \right)} } \right\|_{p,\left[ {a,b} \right]}.
    \end{multline*}

    \item $x = \frac{a+b}{2}$, then we get the following mid-point
    type inequality
    \begin{multline*}
    \left| {\mathcal{R}^{[a,b]}\left(f,u;t_0,\frac{a+b}{2},t_1\right)} \right|
    \\
    \le    L    \left(\frac{b-a}{2}\right)^{1-\frac{1}{p}}  \left( {\frac{{p\sin \left( {\frac{\pi }{p}} \right)}}{{\pi \sqrt[p]{{p - 1}}}}} \right)^n \left\{ {
        \left[ {\frac{{b - a}}{4} + \left| {t_0  - \frac{3a+b}{4}} \right| } \right]^n   }\right.
    \\
    \left.{ +
        \left[ {\frac{{b - a}}{4} + \left| {t_1  - \frac{{a+ 3 b}}{4}} \right| } \right]^n  }\right\}\left\| {f^{\left( n \right)} } \right\|_{p,\left[ {a,b} \right]}.
    \end{multline*}
       For instance, setting $t_0=y$ and $t_1=a+b-y$, we get
    \begin{multline*}
    \left| {\mathcal{R}^{[a,b]}\left(f,u;y,\frac{a+b}{2},a+b-y\right)} \right|
\\
    \le  L  \left(\frac{b-a}{2}\right)^{1-\frac{1}{p}}  \left( {\frac{{p\sin \left( {\frac{\pi }{p}} \right)}}{{\pi \sqrt[p]{{p - 1}}}}} \right)^n \left\{ {
        \left[ {\frac{{b - a}}{4} + \left| {y  - \frac{3a+b}{4}} \right| } \right]^n   }\right.
    \\
    \left.{ +
        \left[ {\frac{{b - a}}{4} + \left| {y  - \frac{{a+ 3 b}}{4}} \right| } \right]^n  }\right\}\left\| {f^{\left( n \right)} } \right\|_{p,\left[ {a,b} \right]}
    \end{multline*}
    for all $y\in \left[{a,\frac{a+b}{2}}\right]$.\\

    \item $t_0=\frac{3a+b}{4}$, $x=\frac{a+b}{2}$ and $t_1=\frac{a+3b}{4}$, then
    \begin{align*}
 \left| {\mathcal{R}^{[a,b]}\left(f,u;\frac{3a+b}{4},\frac{a+b}{2},\frac{a+3b}{2}\right)} \right|
    \le \frac{L}{{2^{n - \frac{1}{p}} }}\left( {b - a} \right)^{n + 1 - \frac{1}{p}}
    \left( {\frac{{p\sin \left( {\frac{\pi }{p}} \right)}}{{\pi \sqrt[p]{{p - 1}}}}} \right)^n  \left\| {f^{\left( n \right)} } \right\|_{p,\left[ {a,b} \right]}
    \end{align*}
 \end{enumerate}
\end{corollary}

%\/\/\/\/\/\/\/\/\/\/\/\/\/\/\/\/\/\/\/\/\/\/\/\/\/\/\/\/\/\/\/\/\/\/\/\/\/\/\/\/\/\/
\section{The dual assumptions}
%\/\/\/\/\/\/\/\/\/\/\/\/\/\/\/\/\/\/\/\/\/\/\/\/\/\/\/\/\/\/\/\/\/\/\/\/\/\/\/\/\/\/
In this section, $L^p$-error estimates of Two-point quadrature
rules for the Riemann--Stieltjes integral $\int_a^b {f\left( t
\right)du\left( t
    \right)}$, where the integrand $f$ is  of bounded variation on $[a,b]$ and the integrator $u$ is assumed to be satisfied the H\"{o}lder condition on $[a,b]$.
\begin{theorem}\label{thm4}
Let $u:[a,b]\rightarrow \mathbb{R}$ be a H\"{o}lder continuous of
order $r$, $(0 < r  \le 1)$, and $f:[a,b]\rightarrow \mathbb{R}$
is a mapping of bounded $p$-variation $(1\le p \le \infty)$ on
$[a, b]$. Then we have the inequality
\begin{multline}
\label{eq3.1}\left|
{\mathcal{R}^{[a,b]}\left(f,u;t_0,x,t_1\right)} \right|
\\
\le H \max \left\{ {\left( {t_0-a} \right), \left[ {\frac{{t_1 -
t_0}}{2} + \left| {x - \frac{{t_0 + t_1}}{2}} \right|}
\right],\left( {b-t_1} \right) } \right\}^r \cdot
\bigvee_a^b\left( {f;p} \right)
\end{multline}
for all  $a\le t_0 \le x \le t_1 \le b$. Furthermore, the constant
$1$ is the best possible in the sense that it cannot be replaced
by a smaller one, for all $r \in (0,1]$.
\end{theorem}

\begin{proof}
Using the integration by parts formula for Riemann--Stieltjes
integral, we have
\begin{align*}
\int_a^{t_0 } {\left[ {u\left( s \right) - u\left( a \right)}
\right]df\left( s \right)}  &= \left[ {u\left( {t_0 } \right) -
u\left( a \right)} \right]f\left( {t_0 } \right) - \int_a^{t_0 }
{f\left( s \right)du\left( s \right)}
\\
\int_{t_0 }^{t_1 } {\left[ {u\left( s \right) - u\left( x \right)}
\right]df\left( s \right)}  &= \left[ {u\left( {t_1 } \right) -
u\left( x \right)} \right]f\left( {t_1 } \right) - \left[ {u\left(
{t_0 } \right) - u\left( x \right)} \right]f\left( {t_0 } \right)
- \int_{t_0 }^{t_1 } {f\left( s \right)du\left( s \right)}
\\
\int_{t_1 }^b {\left[ {u\left( s \right) - u\left( b \right)}
\right]df\left( s \right)}  &= \left[ {u\left( b \right) - u\left(
{t_1 } \right)} \right]f\left( {t_1 } \right) - \int_{t_1 }^b
{f\left( s \right)du\left( s \right)},
\end{align*}
Adding these identities, we get
\begin{align}
\int_a^{t_0 } {\left[ {u\left( s \right) - u\left( a \right)}
\right]df\left( s \right)}  + \int_{t_0 }^{t_1 } {\left[ {u\left(
s \right) - u\left( x \right)} \right]df\left( s \right)}  +
\int_{t_1 }^b {\left[ {u\left( s \right) - u\left( b \right)}
\right]df\left( s \right)}
\nonumber\\
= \left[ {u\left( x \right) - u\left( a \right)} \right]f\left(
{t_0 } \right) + \left[ {u\left( b \right) - u\left( x \right)}
\right]f\left( {t_1 } \right) - \int_a^b {f\left( s
\right)du\left( s \right)}\label{eq3.2}
\end{align}
Applying the triangle inequality on the above identity and then
use Lemma \ref{lemma1}, for each term separately, we get
\begin{align}
&\left| {\left[ {u\left( x \right) - u\left( a \right)}
\right]f\left( {t_0 } \right) + \left[ {u\left( b \right) -
u\left( x \right)} \right]f\left( {t_1 } \right) - \int_a^b
{f\left( s \right)du\left( s \right)}} \right|
\nonumber\\
&\left| {\int_a^{t_0 } {\left[ {u\left( s \right) - u\left( a
\right)} \right]df\left( s \right)}  + \int_{t_0 }^{t_1 } {\left[
{u\left( s \right) - u\left( x \right)} \right]df\left( s \right)}
+ \int_{t_1 }^b {\left[ {u\left( s \right) - u\left( b \right)}
\right]df\left( s \right)}} \right|
\nonumber\\
&\le \left| {\int_a^{t_0 } {\left[ {u\left( s \right) - u\left( a
\right)} \right]df\left( s \right)} } \right| + \left| {\int_{t_0
}^{t_1 } {\left[ {u\left( s \right) - u\left( x \right)}
\right]df\left( s \right)}} \right|+\left| {\int_{t_1 }^b {\left[
{u\left( s \right) - u\left( b \right)} \right]df\left( s
\right)}} \right|
\nonumber\\
&\le \mathop {\sup }\limits_{s \in \left[ {a,t_0} \right]} \left|
{u\left( {s} \right) - u\left( a \right)} \right| \cdot
\bigvee_a^{t_{0}}\left( {f;p} \right)+\mathop {\sup }\limits_{s
\in \left[ {t_0,t_1} \right]} \left| {u\left( {s} \right) -
u\left( x \right)} \right| \cdot \bigvee_{t_{0}}^{t_{1}}\left(
{f;p} \right)\label{eq3.3}
\\
&\qquad+ \mathop {\sup }\limits_{s \in \left[ {t_1,b} \right]}
\left| {u\left( {t_1} \right) - u\left( b \right)} \right| \cdot
\bigvee_{t_1}^b\left( {f;p} \right).\nonumber
\end{align}
As $u$ is of $r$-$H$--H\"{o}lder type, we have
\begin{align*}
\mathop {\sup }\limits_{s \in \left[ {a,t_0} \right]} \left|
{u\left( {s} \right) - u\left( a \right)} \right| \le \mathop
{\sup }\limits_{s \in \left[ {a,t_0} \right]}  \left[ {H\left| {s-
a} \right|^r } \right] = H\left( {t_0-a} \right)^r,
\end{align*}
\begin{align*}
\mathop {\sup }\limits_{s \in \left[ {t_0,t_1} \right]}  \left|
{u\left( {s} \right) - u\left( x \right)} \right| &\le \mathop
{\sup }\limits_{s \in \left[ {t_0,t_1} \right]}\left[ {H\left| {s
- x} \right|^r } \right]
\\
&= H\max \left\{ {\left( {t_1-x} \right)^r ,\left( {x-t_0}
\right)^r } \right\}
\\
&= H\left[ {\max \left\{ {\left( {t_1 - x} \right),\left( {x -
t_0} \right)} \right\}} \right]^r
\\
&= H\left[ {\frac{{t_1 - t_0}}{2} + \left| {x - \frac{{t_0 +
t_1}}{2}} \right|} \right]^r,
\end{align*}
and
\begin{align*}
\mathop {\sup }\limits_{s \in \left[ {t_1,b} \right]}\left|
{u\left( {s} \right) - u\left( b \right)} \right| \le \mathop
{\sup }\limits_{s \in \left[ {t_1,b} \right]}\left[ {H\left| {s-
b} \right|^r } \right] = H\left( {b-t_1} \right)^r,
\end{align*}
Therefore, by \eqref{eq3.3}, we have
\begin{align*}
&\left| {\left[ {u\left( x \right) - u\left( a \right)}
\right]f\left( {t_0 } \right) + \left[ {u\left( b \right) -
u\left( x \right)} \right]f\left( {t_1 } \right) - \int_a^b
{f\left( s \right)du\left( s \right)}} \right|
\\
&\le H\left( {t_0-a} \right)^r \cdot \bigvee_a^{t_{0}}\left( {f;p}
\right)+H\left[ {\frac{{t_1 - t_0}}{2} + \left| {x - \frac{{t_0 +
t_1}}{2}} \right|} \right]^r \cdot \bigvee_{t_{0}}^{t_{1}}\left(
{f;p} \right)+H\left( {b-t_1} \right)^r\cdot \bigvee_{t_1}^b\left(
{f;p} \right)
\\
&\le H \max \left\{ {\left( {t_0-a} \right)^r, \left[ {\frac{{t_1
- t_0}}{2} + \left| {x - \frac{{t_0 + t_1}}{2}} \right|}
\right]^r,\left( {b-t_1} \right)^r } \right\} \cdot
\bigvee_a^b\left( {f;p} \right)
\\
&= H \max \left\{ {\left( {t_0-a} \right), \left[ {\frac{{t_1 -
t_0}}{2} + \left| {x - \frac{{t_0 + t_1}}{2}} \right|}
\right],\left( {b-t_1} \right) } \right\}^r \cdot
\bigvee_a^b\left( {f;p} \right).
\end{align*}
To prove the sharpness of the constant $1$ for any $r \in (0,1]$,
assume that \eqref{eq3.1} holds with a constant $C > 0$, that is,
\begin{multline}
\left| {\left[ {u\left( x \right) - u\left( a
        \right)} \right]f\left( {t_0 } \right) + \left[ {u\left( b \right)
        - u\left( x \right)} \right]f\left( {t_1 } \right) - \int_a^b
    {f\left( s \right)du\left( s \right)}} \right|
\\
\le  C \max \left\{ {\left( {t_0  - a} \right), \left( {
\frac{t_1-t_0}{2} + \left| {x - \frac{{t_0  + t_1 }}{2}} \right|}
    \right), \left( {b - t_1 } \right)} \right\}^r \cdot
\bigvee_{a}^b\left( f;p \right).\label{eq3.4}
\end{multline}
Choose $u(t)= t^r$, $r \in (0,1]$, $t \in [0,1]$ and $f : [0,1]
\to [0,\infty)$ given by
\begin{align*}
f\left( t \right) = \left\{ \begin{array}{l}
0,\,\,\,\,\,\,\,\,\,\,\,t \in \left( {0,1} \right] \\
\\
 1,\,\,\,\,\,\,\,t = 0 \\
\end{array} \right.
\end{align*}
As
\begin{align*}
\left| {u\left( x \right) - u\left( y \right)} \right| = \left|
{x^r  - y^r } \right| \le \left| {x - y} \right|^r, \,\, \forall x
\in [0,1], \,\,r \in (0,1],
\end{align*}
it follows that $u$ is $r$-$H$-H\"{o}lder type with the constant
$H = 1$.

By using the integration by parts formula for Riemann-Stieltjes
integrals, we have:
\begin{align*}
\int_0^1 {f\left( t \right)du\left( t \right)}  = f\left( 1
\right)u\left( 1 \right) - f\left( 0 \right)u\left( 0 \right) -
\int_0^1 {u\left( t \right)df\left( t \right)}=   0,
\end{align*}
and $\bigvee_0^1(f;p) = 1$. Consequently, by \eqref{eq3.4}, we get
\begin{align*}
\left| {t_0^r} \right| \le  C \max \left\{ { t_0  , \left(
{\frac{t_1-t_0}{2} + \left| {x - \frac{{t_0  + t_1 }}{2}} \right|}
    \right), \left( {1 - t_1 } \right)} \right\}^r ,
\,\,\, \forall t_0,t_1\in \left[ {0,1} \right], \,\,
{\rm{with}\,\,t_0\le t_1}.
\end{align*}
Assume first
\begin{align*}
\max \left\{ {t_0  ,\left( {\frac{{t_1
                - t_0 }}{2 } + \left| {x - \frac{{t_0  + t_1 }}{2}} \right|}
    \right),\left( {1 - t_1 } \right)} \right\}^r = t_0^r
\end{align*}
so that we get $1 \le C $.

Now, assume that
\begin{align*}
\max \left\{ {t_0  ,\left( {\frac{{t_1
                - t_0 }}{2 } + \left| {x - \frac{{t_0  + t_1 }}{2}} \right|}
    \right),\left( {1 - t_1 } \right)} \right\}^r = \left( {1 - t_1 } \right)^r.
\end{align*}
choose $t_1 = 1-t_0 $, so that we get $1 \le C $.

Finally, we assume that
\begin{align*}
\max \left\{ {t_0  ,\left( {\frac{{t_1
                - t_0 }}{2 } + \left| {x - \frac{{t_0  + t_1 }}{2}} \right|}
    \right),\left( {1 - t_1 } \right)} \right\}^r = \left( {\frac{{t_1
            - t_0 }}{2 } + \left| {x - \frac{{t_0  + t_1 }}{2}} \right|}
\right)^r.
\end{align*}
Define  $f : [0,1] \to [0,\infty)$ given by
\begin{align*}
f\left( t \right) = \left\{ \begin{array}{l}
0,\,\,\,\,\,\,\,\,\,\,\,t \in \left(  {0,1} \right) \\
\\
1,\,\,\,\,\,\,\,t = 0,1 \\
\end{array} \right.
\end{align*}
Clearly,  $\bigvee_0^1(f;p) = 2$. Therefore,  for $t_0 = 0$ and
$t_1 =1$, so that we get $1 \le C
\left(\frac{1}{2}+\left|x-\frac{1}{2}\right|\right)^r 2^{1/p}$.
Choosing $x=\frac{1}{2}$ and $r=\frac{1}{p}$ or $p=\frac{1}{r}$,
it follows that $1 \le C \left(\frac{1}{2}  \right)^r 2^{r}$,
i.e., $C\ge1$. Hence, the inequality  \eqref{eq3.1}
 is sharp, and the theorem is
completely proved.

\end{proof}

\begin{theorem}
\label{thm5}    Let $1 \le p < \infty$. Let $f,u : [a,b] \to
\mathbb{R}$ be such
    that is $u \in L^p[a,b]$ and $f$ has a Lipschitz property on
    $[a,b]$. If $u$ is $r$-$H$--H\"{o}lder continuous, then the inequality
    \begin{multline}
     \left| {\mathcal{R}^{[a,b]}\left(f,u;t_0,x,t_1\right)} \right|      \\
      \le  L H       \left\{ \begin{array}{l}
     \frac{{\left( {t_0  - a} \right)^{r + 1} }}{{\left( {rp + 1} \right)^{\frac{1}{p}} }} + \left( {t_1  - t_0 } \right)^{1 - \frac{1}{p}} \left( {\frac{{\left( {t_1  - x} \right)^{rp + 1}  - \left( {t_0  - x} \right)^{rp + 1} }}{{rp + 1}}} \right)^{\frac{1}{p}}  + \frac{{\left( {b - t_1 } \right)^{r + 1} }}{{\left( {rp + 1} \right)^{\frac{1}{p}} }},\,\,\,\,\,\,\, a \le x \le t_0 \le t_1 \le b \\
     \\
     \frac{{\left( {t_0  - a} \right)^{r + 1} }}{{\left( {rp + 1} \right)^{\frac{1}{p}} }} + \left( {t_1  - t_0 } \right)^{1 - \frac{1}{p}} \left( {\frac{{\left( {x - t_0 } \right)^{rp + 1}  + \left( {t_1  - x} \right)^{rp + 1} }}{{rp + 1}}} \right)^{\frac{1}{p}}  + \frac{{\left( {b - t_1 } \right)^{r + 1} }}{{\left( {rp + 1} \right)^{\frac{1}{p}} }},\,\,\,\,\,\,\, a\le t_0  \le x \le t_1  \le b\\
     \\
     \frac{{\left( {t_0  - a} \right)^{r + 1} }}{{\left( {rp + 1} \right)^{\frac{1}{p}} }} + \left( {t_1  - t_0 } \right)^{1 - \frac{1}{p}} \left( {\frac{{\left( {x - t_0 } \right)^{rp + 1}  - \left( {x - t_1 } \right)^{rp + 1} }}{{rp + 1}}} \right)^{\frac{1}{p}}  + \frac{{\left( {b - t_1 } \right)^{r + 1} }}{{\left( {rp + 1} \right)^{\frac{1}{p}} }},\,\,\,\,\, \,\, a \le t_0 \le t_1  \le x \le b
     \end{array} \right.
    \end{multline}
    holds for all $p>1$ and $r\in \left( 0,1\right]$ with constant $H>0$.
\end{theorem}
 \begin{proof}
As in the proof of Theorem \ref{thm4}, we have by Lemma
\ref{lemma2}
 \begin{align*}
&\left| {\left[ {u\left( x \right) - u\left( a \right)}
    \right]f\left( {t_0 } \right) + \left[ {u\left( b \right) -
        u\left( x \right)} \right]f\left( {t_1 } \right) - \int_a^b
    {f\left( s \right)du\left( s \right)}} \right|
\nonumber\\
&=\left| {\int_a^{t_0 } {\left[ {u\left( s \right) - u\left( a
            \right)} \right]df\left( s \right)}  + \int_{t_0 }^{t_1 } {\left[
        {u\left( s \right) - u\left( x \right)} \right]df\left( s \right)}
    + \int_{t_1 }^b {\left[ {u\left( s \right) - u\left( b \right)}
        \right]df\left( s \right)}} \right|
\nonumber\\
&\le \left| {\int_a^{t_0 } {\left[ {u\left( s \right) - u\left( a
            \right)} \right]df\left( s \right)} } \right| + \left| {\int_{t_0
    }^{t_1 } {\left[ {u\left( s \right) - u\left( x \right)}
        \right]df\left( s \right)}} \right|+\left| {\int_{t_1 }^b {\left[
        {u\left( s \right) - u\left( b \right)} \right]df\left( s
        \right)}} \right|
\nonumber\\
&\le L\left[ {\left( {t_0  - a} \right)^{1 - \frac{1}{p}} \left(
{\int_a^{t_0 } {\left| {u\left( s \right) - u\left( a \right)}
\right|^p ds} } \right)^{\frac{1}{p}}  + \left( {t_1  - t_0 }
\right)^{1 - \frac{1}{p}} \left( {\int_{t_0 }^{t_1 } {\left|
{u\left( s \right) - u\left( x \right)} \right|^p ds} }
\right)^{\frac{1}{p}} } \right.
\\
&\qquad\left. { + \left( {b - t_1 } \right)^{1 - \frac{1}{p}}
\left( {\int_{t_1 }^b {\left| {u\left( s \right) - u\left( b
\right)} \right|^p ds} } \right)^{\frac{1}{p}} } \right]
\\
&\le L H\left[ {\left( {t_0  - a} \right)^{1 - \frac{1}{p}} \left(
{\int_a^{t_0 } {\left| {s-a} \right|^{rp} ds} }
\right)^{\frac{1}{p}}  + \left( {t_1  - t_0 } \right)^{1 -
\frac{1}{p}} \left( {\int_{t_0 }^{t_1 } {\left| {s-x} \right|^{rp}
ds} } \right)^{\frac{1}{p}} } \right.
\\
&\qquad\left. { + \left( {b - t_1 } \right)^{1 - \frac{1}{p}}
\left( {\int_{t_1 }^b {\left| {s-b} \right|^{rp} ds} }
\right)^{\frac{1}{p}} } \right].
\end{align*}
Simple computations yield that
\begin{align*}
\int_a^{t_0 } {\left| {s - a} \right|^{rp} ds}  = \int_a^{t_0 }
{\left( {s - a} \right)^{rp} ds}  = \frac{{\left( {t_0  - a}
\right)^{rp + 1} }}{{rp + 1}},
\end{align*}
\begin{align*}
\int_{t_0 }^{t_1 } {\left| {s - x} \right|^{rp} ds}  &= \left\{
\begin{array}{l}
\int_{t_0 }^{t_1 } {\left( {s - x} \right)^{rp} ds}, \,\,\,\,\,\,\,\,\,\,\,\,\,\,\,\,\,\,\,\,\,\,\,\,\,\,\,\,\,\,\,\,\,\,\,\,\,\,\,\,\,\,\,a \le x \le t_0  \\
\\
\int_{t_0 }^x {\left( {x - s} \right)^{rp} ds}  + \int_x^{t_1 } {\left( {s - x} \right)^{rp} ds} ,\,\,\,\,\,\,\,\,t_0  \le x \le t_1  \\
\\
\int_{t_0 }^{t_1 } {\left( {x - s} \right)^{rp} ds} ,\,\,\,\,\,\,\,\,\,\,\,\,\,\,\,\,\,\,\,\,\,\,\,\,\,\,\,\,\,\,\,\,\,\,\,\,\,\,\,\,\,\,\,t_1  \le x \le b \\
\end{array} \right. \\
&= \left\{ \begin{array}{l}
\frac{{\left( {t_1  - x} \right)^{rp + 1}  - \left( {t_0  - x} \right)^{rp + 1} }}{{rp + 1}}\,\,\,\,\,\,\,\,\,\,a \le x \le t_0  \\
\\
\frac{{\left( {x - t_0 } \right)^{rp + 1}  + \left( {t_1  - x} \right)^{rp + 1} }}{{rp + 1}},\,\,\,\,\,\,\,\,t_0  \le x \le t_1  \\
\\
\frac{{\left( {x - t_0 } \right)^{rp + 1}  - \left( {x - t_1 } \right)^{rp + 1} }}{{rp + 1}},\,\,\,\,\,\,\,\,t_1  \le x \le b \\
\end{array} \right.,\\
\end{align*}
and
\begin{align*}
\int_{t_1 }^b {\left| {s - b} \right|^{rp} ds}  = \int_{t_1 }^b
{\left( {b - s} \right)^{rp} ds}  = \frac{{\left( {b - t_1 }
\right)^{rp + 1} }}{{rp + 1}}.
\end{align*}
Combining these equalities with the last inequality above we get
the required result.
\end{proof}

\begin{corollary}
\label{cor5}    Let $1 \le p < \infty$. Let $f,u : [a,b] \to
\mathbb{R}$ be such
    that is $u \in L^p[a,b]$ and $f$ has a Lipschitz property on
    $[a,b]$. If $u$ is $r$-$H$--H\"{o}lder continuous, then the inequality
    \begin{multline}
\label{eq3.6}   \left| {\left( {x-a}
        \right)f\left( {t_0 } \right) + \left( {b-x} \right)f\left( {t_1 } \right) - \int_a^b
        {s^{r-1}f\left( s \right)ds}} \right|    \\
    \le  L H     \left\{ \begin{array}{l}
    \frac{{\left( {t_0  - a} \right)^{r + 1} }}{{\left( {rp + 1} \right)^{\frac{1}{p}} }} + \left( {t_1  - t_0 } \right)^{1 - \frac{1}{p}} \left( {\frac{{\left( {t_1  - x} \right)^{rp + 1}  - \left( {t_0  - x} \right)^{rp + 1} }}{{rp + 1}}} \right)^{\frac{1}{p}}  + \frac{{\left( {b - t_1 } \right)^{r + 1} }}{{\left( {rp + 1} \right)^{\frac{1}{p}} }},\,\,\,\,\,\,\, a \le x \le t_0 \le t_1 \le b \\
    \\
    \frac{{\left( {t_0  - a} \right)^{r + 1} }}{{\left( {rp + 1} \right)^{\frac{1}{p}} }} + \left( {t_1  - t_0 } \right)^{1 - \frac{1}{p}} \left( {\frac{{\left( {x - t_0 } \right)^{rp + 1}  + \left( {t_1  - x} \right)^{rp + 1} }}{{rp + 1}}} \right)^{\frac{1}{p}}  + \frac{{\left( {b - t_1 } \right)^{r + 1} }}{{\left( {rp + 1} \right)^{\frac{1}{p}} }},\,\,\,\,\,\,\, a\le t_0  \le x \le t_1  \le b\\
    \\
    \frac{{\left( {t_0  - a} \right)^{r + 1} }}{{\left( {rp + 1} \right)^{\frac{1}{p}} }} + \left( {t_1  - t_0 } \right)^{1 - \frac{1}{p}} \left( {\frac{{\left( {x - t_0 } \right)^{rp + 1}  - \left( {x - t_1 } \right)^{rp + 1} }}{{rp + 1}}} \right)^{\frac{1}{p}}  + \frac{{\left( {b - t_1 } \right)^{r + 1} }}{{\left( {rp + 1} \right)^{\frac{1}{p}} }},\,\,\,\,\, \,\, a \le t_0 \le t_1  \le x \le b
    \end{array} \right.
    \end{multline}
    holds for all $p>1$ and $r\in \left( 0,1\right]$ with constant $H>0$.
\end{corollary}
\begin{proof}
Setting $u(t)=t^r$, $t\in [a,b]$, $r\in (0,1]$,  in Theorem
\ref{thm5} we get the required result.
\end{proof}

\begin{corollary}
    Let $1 \le p < \infty$. Let $f,u : [a,b] \to \mathbb{R}$ be such
    that is $u \in L^p[a,b]$ and $f$ has a Lipschitz property on
    $[a,b]$. If $u$ is $K$-Lipschitz continuous on $[a,b]$, then the inequality
    \begin{multline}
\label{eq3.7}   \left| {\left( {x-a}
        \right)f\left( {t_0 } \right) + \left( {b-x} \right)f\left( {t_1 } \right) - \int_a^b
        {f\left( s \right)ds}} \right|   \\
    \le  L K     \left\{ \begin{array}{l}
    \frac{{\left( {t_0  - a} \right)^{2} }}{{\left( {p + 1} \right)^{\frac{1}{p}} }} + \left( {t_1  - t_0 } \right)^{1 - \frac{1}{p}} \left( {\frac{{\left( {t_1  - x} \right)^{p + 1}  - \left( {t_0  - x} \right)^{p + 1} }}{{p + 1}}} \right)^{\frac{1}{p}}  + \frac{{\left( {b - t_1 } \right)^{2} }}{{\left( {p + 1} \right)^{\frac{1}{p}} }},\,\,\,\,\,\,\, a \le x \le t_0 \le t_1 \le b \\
    \\
    \frac{{\left( {t_0  - a} \right)^{2} }}{{\left( {p + 1} \right)^{\frac{1}{p}} }} + \left( {t_1  - t_0 } \right)^{1 - \frac{1}{p}} \left( {\frac{{\left( {x - t_0 } \right)^{p + 1}  + \left( {t_1  - x} \right)^{p + 1} }}{{p + 1}}} \right)^{\frac{1}{p}}  + \frac{{\left( {b - t_1 } \right)^{2} }}{{\left( {p + 1} \right)^{\frac{1}{p}} }},\,\,\,\,\,\,\, a\le t_0  \le x \le t_1  \le b\\
    \\
    \frac{{\left( {t_0  - a} \right)^{2} }}{{\left( {p + 1} \right)^{\frac{1}{p}} }} + \left( {t_1  - t_0 } \right)^{1 - \frac{1}{p}} \left( {\frac{{\left( {x - t_0 } \right)^{rp + 1}  - \left( {x - t_1 } \right)^{p + 1} }}{{rp + 1}}} \right)^{\frac{1}{p}}  + \frac{{\left( {b - t_1 } \right)^{2} }}{{\left( {p + 1} \right)^{\frac{1}{p}} }},\,\,\,\,\, \,\, a \le t_0 \le t_1  \le x \le b
    \end{array} \right.
    \end{multline}
    holds for all $p>1$ and   constant $K>0$.
\end{corollary}
\begin{proof}
    Setting $r=1$c  in Corollary \ref{cor5}, we get the required result.
\end{proof}

\begin{remark}
    The inequalities \eqref{eq3.6} and \eqref{eq3.7} generalize the recent  result(s) in \cite{A}.
\end{remark}

\centerline{}

\centerline{}


\begin{thebibliography}{9}
 \setlength{\itemsep}{5pt}

\bibitem{AG} M.W. Alomari  and A. Guessab, $L^p$--error bounds of two and three--point quadrature rules for Riemann--Stieltjes inegrals, {\em Moroccan J. Pure \& Appl.
    Anal.}  (MJPAA),  accepted.

\bibitem{A} M.W. Alomari, Two-point Ostrowski's inequality, Results in Mathematics, {\bf 72} (3) (2017),
 1499--1523.

\bibitem{alomari}
M.W. Alomari, On Beesack--Wirtinger inequality, \textit{Results in
Mathematics},  {\bf 72} (3) (2017), 1213--1225.

\bibitem{alomari1}

M.W. Alomari and S.S. Dragomir, Mercer-Trapezoid rule for
Riemann--Stieltjes integral with applications, \textit{Journal of
Advances in Mathematics}, 2 (2) (2013), 67--85.

\bibitem{alomari2}
M.W. Alomari, A companion of Ostrowski's inequality for the
Riemann-Stieltjes integral $\int_{a}^{b}{f\left( t\right) du\left(
t\right) }$, where $f$ is of bounded variation and $u$ is of
$r$-$H$-H\"{o}lder type and applications, \textit{Appl. Math.
Comput.}, 219 (2013), 4792--4799.

\bibitem{alomari3} M.W. Alomari, New sharp inequalities of Ostrowski and
generalized trapezoid type for the Riemann--Stieltjes integrals
and applications, \textit{Ukrainian Mathematical Journal}, 65 (7)
2013, 895--916.


\bibitem{alomari4} M.W. Alomari, Approximating the Riemann-Stieltjes integral by a
three-point quadrature rule and applications, \textit{Konuralp J.
Math.}, 2 (2) (2014), 22?34.


\bibitem{M} M.W. Alomari, Two point Gauss-Legendre quadrature rule for Riemann-Stieltjes
integrals, Preprint (2014). Avaliable at
\url{https://arxiv.org/pdf/1402.4982.pdf}

\bibitem{alomari5}
M.W. Alomari, A sharp companion of Ostrowski's inequality for the
Riemann--Stieltjes integral and applications, \textit{Ann. Univ.
Paedagog. Crac. Stud. Math.}, 15 (2016), 69--78.

\bibitem{alomari6} M.W. Alomari and S.S. Dragomir, A three-point quadrature rule for
the Riemann-Stieltjes integral, \textit{Southeast Asian Bulletin
Journal of Mathematics}, in press

\bibitem{Barnett}
N.S. Barnett, S.S. Dragomir and I. Gomma,  A companion for the
Ostrowski and the generalised trapezoid inequalities, \textit{
Mathematical and Computer Modelling,} 50 (2009), 179--187.

\bibitem{Barnett1}
N.S. Barnett, W.-S. Cheung, S.S. Dragomir, A. Sofo,  Ostrowski and
trapezoid type inequalities for the Stieltjes integral with
Lipschitzian integrands or integrators, \textit{Computer \&
Mathematics with  Applications}, 57 (2009), 195--201.


\bibitem{CeroneDragomir}
P. Cerone, S.S. Dragomir, New bounds for the three-point rule
involving the Riemann-Stieltjes integrals, in: C. Gulati, et al.
(Eds.), Advances in Statistics Combinatorics and Related Areas,
World Science Publishing, 2002, pp. 53--62.

\bibitem{CeroneDragomir1}
P. Cerone, S.S. Dragomir, Approximating the Riemann--Stieltjes
integral via some moments of the integrand, \textit{Mathematical
and Computer Modelling}, 49 (2009), 242--248.



\bibitem{Dragomir1}
S.S. Dragomir, On the Ostrowski inequality for Riemann--Stieltjes
integral $\int_a^b{f(t) du(t)}$ where $f$ is of H\"{o}lder type
and $u$ is of bounded variation and applications, {\em J. KSIAM,}
5 (2001), 35--45.

\bibitem{Dragomir2}
S.S. Dragomir, On the Ostrowski's inequality for Riemann-Stieltes
integral and applications, \textit{Korean J. Comput. \& Appl.
Math.}, 7 (2000),  611--627.



\bibitem{Dragomir5}
S.S. Dragomir, C. Bu\c{s}e, M.V. Boldea, L. Braescu, A
generalisation of the trapezoid rule for the Riemann-Stieltjes
integral and applications, \textit{Nonlinear Anal. Forum} 6 (2)
(2001) 33--351.

\bibitem{Dragomir6}
S.S. Dragomir, Some inequalities of midpoint and trapezoid type
for the Riemann-Stieltjes integral, \textit{Nonlinear Anal.} 47
(4) (2001) 2333--2340.

\bibitem{Dragomir7}
S.S. Dragomir, Approximating the Riemann-Stieltjes integral in
terms of generalised trapezoidal rules, \textit{Nonlinear Anal.}
TMA 71 (2009) e62--e72.

\bibitem{Dragomir8}
S.S. Dragomir, Approximating the Riemann-Stieltjes integral by a
trapezoidal quadrature rule with applications,
\textit{Mathematical and Computer Modelling} 54 (2011) 243--260.

\bibitem{Gautschi}
W. Gautschi, On generating orthogonal polynomials, \text{SIAM J.
Sci. Stat. Comput.}, 3 (1982), 289--317.

\bibitem{Guessab}
A. Guessab and G. Schmeisser, Sharp integral inequalities of the
Hermite-Hadamard type, \textit{J. Approx. Th.}, 115 (2002),
260--288.






\bibitem{Dudley}R. M. Dudley, Frechet Differentiability, $p$-variation and uniform Donsker
classes, {\em The Annals of Probability\/},~{\bf 20} (4) (1992),
1968--1982.


\bibitem{Golubov1}
B.I. Golubov, On criteria for the continuity of functions of
bounded $p$-variation, {\em Sibirskii Matematicheskii
Zhurnal\/},~{\bf 13} (5) (1972), 1002?--1015.

\bibitem{Golubov2}
B.I. Golubov, On functions of bounded $p$-variation, {\em
Mathematics of the USSR-Izvestiya\/},~{\bf 2} (4) (1968),
799?--819.


\bibitem{Kolyada}
V.I. Kolyada and M. Lind, On functions of bounded $p$-variation,
{\em J.  Math. Anal. Appl.\/},~{\bf 356} (2009), 582?--604.



\bibitem{Mercer}
P.R. Mercer, Hadamard's inequality and trapezoid rules for the
Riemann?Stieltjes integral, \textit{J. Math. Anal. Appl.} 344
(2008) 921--926.

\bibitem{Munteanu}M. Munteanu, Quadrature formulas for the generalized
Riemann-Stieltjes integral, Bull. Braz. Math. Soc. (N.S.) 38 (1)
(2007) 39?-50.

\bibitem{Natanson}
I.P. Natanson, Theory of Functions of a Real Variable, Vol I,
Translated from the Russian by L.F. Boron and E. Hewitt, Frederick
Ungar Publishing, New York, 1955.


\bibitem{Wiener}
N. Wiener, The quadratic variation of a function and its Fourier
coefficients, {\em Massachusetts J. Math.\/},~{\bf 3} (1924),
72--94.

\bibitem{Young}L.C. Young, An inequality of the H\"{o}lder type, connected with
Stieltjes integration, {\em Acta Math.\/},~{\bf 67}  (1936),
251--282.


\bibitem{Tortorella} M. Tortorella, Closed Newton-Cotes quadrature
rules for Stieltjes integrals and numerical convolution of life
distributions, SIAM J. Sci. Stat. Comput. 11 (1990) 732--748.



\end{thebibliography}
\end{document}